\documentclass[a4paper,11pt]{amsart}
\usepackage{amsmath,amssymb,amsthm,mathtools,hyperref}
\allowdisplaybreaks
\numberwithin{equation}{section}
\theoremstyle{plain}
\newtheorem{thm}{Theorem}

\newtheorem{prop}[thm]{Proposition}
\newtheorem{lem}{Lemma}
\newtheorem{cor}[thm]{Corollary}

\newtheorem{conj}[thm]{Conjecture}

\newtheorem{rem}{Remark}

\newtheorem*{acknow}{Acknowledgments}

\makeatletter

\def\l{\langle}
\def\r{\rangle}
\def\tri{\triangle}
\def\nnn{\nabla}

\makeatother

\title[Pinching rigidity of minimal surfaces in spheres]{Pinching rigidity of minimal surfaces in spheres}

\author[W.R. Ding]{Weiran Ding$^{1}$}
\address{$^{1}$School of Mathematical Sciences, Laboratory of Mathematics and Complex Systems, Beijing Normal University, Beijing 100875, P. R. CHINA.}
\email{dingwr0806@mail.bnu.edu.cn}

\author[J.Q. Ge]{Jianquan Ge$^{2,*}$}
\address{$^{2,*}$School of Mathematical Sciences, Laboratory of Mathematics and Complex Systems, Beijing Normal University, Beijing 100875, P. R. CHINA.}
\email{jqge@bnu.edu.cn}

\author[F.G. Li]{Fagui Li$^{3,4}$}
\address{$^{3}$Greater Bay Area Innovation Research Institute, Beijing Institute of Technology, Zhuhai, Guangdong 519088, P. R. CHINA.}
\address{$^{4}$Beijing International Center for Mathematical Research, Peking University, Beijing 100871, P. R. CHINA.}
\email{faguili@bicmr.pku.edu.cn}

\subjclass[2010]{53C24, 53C42, 53C65.}
\date{}
\keywords{minimal surfaces; rigidity theorem; the Simon conjecture}
\thanks{$^{*}$ the corresponding author.}
\thanks{J. Q. Ge is partially supported by NSFC (No. 12171037) and the Fundamental Research Funds for the Central Universities.}
\thanks{F. G. Li is partially supported by  NSFC (No. 12171037, 12271040) and China Postdoctoral Science Foundation (No. 2022M720261).}

\begin{document}
\maketitle

\begin{abstract}
In 1980, U. Simon proposed a quantization conjecture about the Gaussian curvature $K$ of closed minimal surfaces in unit spheres: if $K(s+1)\leq K\leq K(s)$ ($K(s)\coloneqq2/(s(s+1))$, $s\in \mathbb{N}$), then either $K=K(s)$ or $K=K(s+1)$. Notice that the surface must be one of Calabi's standard minimal $2$-spheres if the curvature is a positive constant. The cases $s=1$ and $s=2$ were proven in the 1980s by Simon and others. In this paper we give a pinching theorem of the Simon conjecture in the case $s=3$ and also give a new proof of the cases $s=1$ and $s=2$ by some Simons-type integral inequalities.
\end{abstract}

\section{Introduction}\label{s2}

In 1967, E. Calabi~(cf. \cite{3}) studied minimal immersions of $\mathbb{S}^2$ with constant Gaussian curvature $K$ into $\mathbb{S}^N(1)$. These immersions were classified up to a rigid motion with the curvature $K$ corresponding to the following values, $K=K(s)\coloneqq\frac{2}{s(s+1)}, s\in \mathbb{N}$~(cf. \cite{9}). In 1980, U. Simon made the following quantization conjecture~(cf. \cite{13, 17, 25}).

\begin{conj}[Intrinsic version]
Let $M$ be a closed surface minimally immersed into $\mathbb{S}^N(1)$ such that the image is not contained in any hyperplane of $\mathbb{R}^{N+1}$. If $K(s+1)\leq K\leq K(s)$ for an $s\in \mathbb{N}$, then either $K=K(s+1)$ or $K=K(s)$, and thus the immersion is one of the Calabi's standard minimal immersions with the dimension of the ambient space $N=2s+2$ or $N=2s$.
\end{conj}

For minimal surfaces in $\mathbb{S}^N(1)$, the curvature $K$ and the squared norm $S=|h|^2$ of the second fundamental form $h$ are related as follows:
\begin{equation*}\label{KSeq}
2K=2-S.
\end{equation*}
It follows that, by setting
\begin{equation*}
S(s)\coloneqq\frac{2(s-1)(s+2)}{s(s+1)}=2-2K(s),
\end{equation*}
the Simon conjecture above can also be stated as:

\begin{conj}[Extrinsic version]
Let $M$ be a closed surface minimally immersed into $\mathbb{S}^N(1)$ such that the image is not contained in any hyperplane of $\mathbb{R}^{N+1}$. If $S(s)\leq S\leq S(s+1)$ for an $s\in\mathbb{N}$, then either $S=S(s)$ or $S=S(s+1)$, and thus the immersion is one of the Calabi's standard minimal immersions with the dimension of the ambient space $N=2s$ or $N=2s+2$.
\end{conj}

So far, the Simon conjecture has only been solved in the cases $s=1$ and $s=2$~(cf. \cite{1, 13}). Remarkably, without the assumption of minimality of the immersion, Li and Simon gave a generalization of the Simon conjecture in the case $s=1$~(cf. \cite{17}). In fact, the case $s=1$ has also been studied in the extrinsic version even for higher-dimensional minimal submanifolds by Simons-type integral inequalities~(cf. \cite{6, 14, 16, 18, 26}, etc.). It shows the pinching rigidity and the first gap of the squared norm $S$ of the second fundamental form, whereas the case $s=2$ shows the second gap as the well-known Peng-Terng-type second gap theorems for hypersurfaces in the unit sphere~(cf. \cite{5, 8, 15, 21, 31, 32}, etc.). More generally, the discrete property of $S$ of higher-dimensional minimal hypersurfaces in the unit sphere is linked to the famous Chern Conjecture~(cf. \cite{4, 27, 28, 29, 33}, etc.). Please see the excellent and detailed surveys for more  developments and references on this type of rigidity problems~(cf. \cite{10, 24}). To the best of our knowledge, the third gap of the Simon conjecture  is far from being solved, although there are indeed many partial results for the case $s\geq3$ of the Simon conjecture under additional assumptions~(cf. \cite{2, 7, 12, 19, 20, 23}, etc.). 

In this paper, we make progress on the case $s=3$ of the Simon conjecture in its extrinsic version by establishing Simons-type integral inequalities. Specifically, we prove the following result:

\setcounter{thm}{0}
\renewcommand{\thethm}{\Alph{thm}}
\begin{thm}\label{maintheorem}
Let $M$ be a closed surface minimally immersed into $\mathbb{S}^N(1)$.
\begin{itemize}
\item [\emph{(i)}] If $0\leq S\leq \frac{4}{3}$, then $S=0$ or $S=\frac{4}{3}$;
\item [\emph{(ii)}] If $\frac{4}{3}\leq S\leq\frac{5}{3},$ then $S=\frac{4}{3}$ or $S=\frac{5}{3}$;
\item [\emph{(iii)}] If $\frac{5}{3}\le S\le \frac{9}{5}$,  $S_{\max}=\sup_{p\in M}S(p)$ and $ S_{\min}=\inf_{p\in M}S(p)$, then $$S_{\max}-S_{\min}\ge\frac{134-114S_{\min}+\sqrt{\mathcal{F}}}{108},$$ where $\mathcal{F}=(134-114S_{\min})^2+864(3S_{\min}-5)(9-5S_{\min})$.
\end{itemize}
\end{thm}

\begin{rem}
By choosing a special frame field on $M$, Yang~\emph{(}cf. \emph{\cite{30})} provided a proof of the cases $s=1,2$ if $M$ has a flat or nowhere flat normal bundle. We now show that the ideas there work without the assumption on the flatness of the normal bundle. 
\end{rem}

\begin{rem}
Okayasu~\emph{(}cf. \emph{\cite{19})} proved a similar result to Theorem \ref{maintheorem} by using a different method.
\end{rem}
By Theorem \ref{maintheorem} (iii), one has
\begin{cor}\label{corollary 3 gap}
Let $M$ be a closed $2$-dimensional Riemannian manifold. Then there exists no isometric minimal immersion $\Phi: M\to \mathbb{S}^N$ for any $N$ such that $\frac{5}{3}\le S_{\min} \le \frac{9}{5}$ and $$S_{\max}<S_{\min}+\frac{134-114S_{\min}+\sqrt{\mathcal{F}}}{108},$$ where $\mathcal{F}=(134-114S_{\min})^2+864(3S_{\min}-5)(9-5S_{\min})$.
\end{cor}
\begin{rem}
Let $$f(S_{\min})=\frac{134-114S_{\min}+\sqrt{\mathcal{F}}}{108},\quad\frac{5}{3}\le S_{\min}\le\frac{9}{5}.$$ It should be pointed out that we cannot obtain any gaps if $S_{\min}=\frac{5}{3}$~{or}~$\frac{9}{5}$, and the gap exists when $\frac{5}{3}<S_{\min}<\frac{9}{5}$. Now,
\begin{equation*}
f'(S_{\min})=\frac{1}{108}\left(\frac{6(3S_{\min}+599)}{\sqrt{9S_{\min}^2+3594S_{\min}-5231}}-114\right),\quad\frac{5}{3}\le S_{\min}\le\frac{9}{5}.
\end{equation*}
Letting $f'(S_{\min})=0$, we obtain $-45(S_{\min})^2-17970S_{\min}+31211=0$, from which we deduce that $S_{\min}\approx1.72935007\in(\frac{5}{3},\frac{9}{5})$, at which $f(S_{\min})$ attains its maximal value $(f(S_{\min}))_{\max}\approx0.00419291$. Hence, there exists no isometric minimal immersion $\Phi: M\to\mathbb{S}^N$ for any $N$ such that $1.72936<S<1.73355$.
\end{rem}
In addition, we also obtain the following integral equations.
\begin{thm}\label{theorem integral equations}
Let $M$ be a closed minimal surface immersed in $\mathbb{S}^N(1)$ with positive Gaussian curvature. 
Then we have some Simons-type identities  as follows (see Theorems \ref{thm gap s=1}, \ref{thm gap s=2} and \ref{thm An integral inequality}):
\begin{align*}
\int_M S(3S-4)&=2\int_M\mathcal{B}_1\ge0,\\
\int_M S(3S-4)(3S-5)&=2\int_M[\mathcal{B}_2-\frac{1}{4}S(3S-4)^2+\frac{1}{2}|\nabla S|^2]\ge0,\\
\int_M S(3S-4)(3S-5)(5S-9)&=2\int_M[\mathcal{B}_3-\frac{1}{8}S(3S-4)(45S^2-144S+116)\\
&\hspace{4em}+\frac{1}{8}(65S-166)|\nabla S|^2-\frac{5}{8}(\triangle S)^2],
\end{align*}
where $\mathcal{B}_1=|\nabla h|^2$, $\mathcal{B}_2=|\nabla^2 h|^2$ and $\mathcal{B}_3=|\nabla^3 h|^2$  are the squared lengths of the first, second and third covariant derivatives of $h$, respectively. 
\end{thm}
\setcounter{thm}{0}
\renewcommand{\thethm}{\arabic{thm}}

\section{Notations and Local formulas}

Let $M$ be a $2$-dimensional manifold immersed in a unit sphere $\mathbb{S}^N(1)$. We assume the range of the indices as follows: $$1\leq i,j,k,\cdots\leq 2;\quad 3\leq\alpha,\beta,\gamma,\cdots\leq N;\quad 1\leq A,B,C,\cdots\leq N.$$ Let $\{e_1,...,e_N\}$ be a local orthonormal frame on $T(\mathbb{S}^N(1))$ such that, when restricted to $M$, $\{e_1,e_2\}$($\{e_3,...,e_N\}$) lie in the tangent bundle $T(M)$ (normal bundle $T^{\bot}(M)$). We take ($\omega_A$) and ($\omega_{AB}$) as the metric 1-form field and connection form field associated with $\{e_1,...,e_N\}$. Let $S_\alpha=(h_{ij}^\alpha)_{2\times 2}$, where $\omega_{i\alpha}=h_{ij}^\alpha\omega_j$. Then, $h_{ij}^\alpha=h_{ji}^\alpha$. In the following, we will use the Einstein summation convention. The second fundamental form of $M$ is defined by $h=h_{ij}^\alpha \omega_i\omega_j e_\alpha.$ The mean curvature normal vector field is defined by $H=\frac{1}{2}h_{ii}^\alpha e_\alpha$. If the mean curvature normal vector field of $M$ vanishes identically, the immersion is called minimal. Now we consider minimal surfaces. Define column vectors $a\coloneqq(a^{\alpha})\in\mathbb{R}^{p}, b\coloneqq(b^\alpha)\in\mathbb{R}^{p}$, where $a^\alpha\coloneqq h_{11}^\alpha=-h_{22}^\alpha,b^\alpha\coloneqq h_{12}^\alpha=h_{21}^\alpha$ and $p=N-2$ is the codimention. We use the following notations: $$A\coloneqq(\l S_{\alpha},S_{\beta}\r)=2aa^{\text{T}}+2bb^{\text{T}},\quad S\coloneqq\text{tr}A=|h|^2,\quad\rho^{\bot}\coloneqq\sum_{\alpha,\beta}|[S_{\alpha},S_{\beta}]|^2.$$ The covariant derivatives $h_{ijk}^\alpha$, $h_{ijkl}^\alpha$, $h_{ijklm}^\alpha$ and $h_{ijklmn}^\alpha$ are defined as follows:
\begin{align*}
h_{ijk}^\alpha\omega_k&=dh_{ij}^\alpha+h_{mj}^\alpha\omega_{mi}+h_{im}^\alpha\omega_{mj}+h_{ij}^\beta\omega_{\beta\alpha},\\
h_{ijkl}^\alpha\omega_l&=dh_{ijk}^\alpha+h_{mjk}^\alpha\omega_{mi}+h_{imk}^\alpha\omega_{mj}+h_{ijm}^\alpha\omega_{mk}+h_{ijk}^\beta\omega_{\beta\alpha},\\
h_{ijklm}^\alpha\omega_m&=dh_{ijkl}^\alpha+h_{njkl}^\alpha\omega_{ni}+h_{inkl}^\alpha\omega_{nj}+h_{ijnl}^\alpha\omega_{nk}+h_{ijkn}^\alpha\omega_{nl}+h_{ijkl}^\beta\omega_{\beta\alpha},\\
h_{ijklmn}^\alpha\omega_n&=dh_{ijklm}^\alpha+h_{pjklm}^\alpha\omega_{pi}+h_{ipklm}^\alpha\omega_{pj}+h_{ijplm}^\alpha\omega_{pk}+h_{ijkpm}^\alpha\omega_{pl}+h_{ijklp}^\alpha\omega_{pm}+h_{ijklm}^\beta\omega_{\beta\alpha}.
\end{align*}
For convenience, the following notations are defined by
\begin{equation*}
a_i^{\alpha}\coloneqq h_{11i}^{\alpha},~a_i\coloneqq (a_i^\alpha),
\end{equation*}
and
\begin{equation*}
\mathcal{B}_1\coloneqq\sum_{i,j,k,\alpha}(h_{ijk}^\alpha)^2,~\mathcal{B}_2\coloneqq\sum_{i,j,k,l,\alpha}(h_{ijkl}^\alpha)^2,~\mathcal{B}_3\coloneqq\sum_{i,j,k,l,m,\alpha}(h_{ijklm}^\alpha)^2.
\end{equation*}
The Codazzi equation and Ricci's formulas are
\begin{align}
h_{ijk}^\alpha-h_{ikj}^\alpha&=0,\label{Cod}\\
h_{ijkl}^\alpha-h_{ijlk}^\alpha&=h_{pj}^\alpha R_{pikl}+h_{ip}^\alpha R_{pjkl}+h_{ij}^\beta R_{\beta\alpha kl},\label{Ric1}\\
h_{ijklm}^\alpha-h_{ijkml}^\alpha&=h_{pjk}^\alpha R_{pilm}+h_{ipk}^\alpha R_{pjlm}+h_{ijp}^\alpha R_{pklm}+h_{ijk}^\beta R_{\beta\alpha lm},\label{Ric2}\\
h_{ijklmn}^\alpha-h_{ijklnm}^\alpha&=h_{pjkl}^\alpha R_{pimn}+h_{ipkl}^\alpha R_{pjmn}+h_{ijpl}^\alpha R_{pkmn}\\
&\hspace{2em}+h_{ijkp}^\alpha R_{plmn}+h_{ijkl}^\beta R_{\beta\alpha mn}.\label{Ric3}
\end{align}
The Laplacians $\triangle h_{ij}^\alpha$, $\triangle h_{ijk}^\alpha$ and $\triangle h_{ijkl}^\alpha$ are defined by
\begin{equation*}
\tri h_{ij}^\alpha=\sum_{k}h_{ijkk}^\alpha,~\tri h_{ijk}^\alpha=\sum_{l}h_{ijkll}^\alpha,~\tri h_{ijkl}^\alpha=\sum_{m}h_{ijklmm}^\alpha.
\end{equation*}
From \eqref{Cod}, \eqref{Ric1} and \eqref{Ric2}, we obtain
\begin{equation}\label{Lap1}
\tri h_{ij}^\alpha =h_{mmij}^\alpha+h_{pi}^\alpha R_{pmjm}+h_{mp}^\alpha R_{pijm}+h_{mi}^\delta R_{\delta\alpha jm},
\end{equation}
and
\begin{equation}\label{Lap2}
\begin{aligned}
\tri h_{ijk}^\alpha &=(\tri h_{ij}^\alpha)_k+2h_{pjm}^\alpha R_{pikm}+2h_{ipm}^\alpha R_{pjkm}+h_{ijp}^\alpha R_{pmkm}+2h_{ijm}^\delta R_{\delta\alpha km}\\
&\hspace{1.3em}+h_{pj}^\alpha R_{pikmm}+h_{ip}^\alpha R_{pjkmm}+h_{ij}^\delta R_{\delta\alpha kmm}.
\end{aligned}	
\end{equation}
The Riemannian curvature tensor, the normal curvature tensor and the first covariant differentials of the normal curvature tensor are given by
\begin{align}
R_{ijkl}&=\frac{1}{2}(2-S)(\delta_{ik}\delta_{jl}-\delta_{il}\delta_{jk}),\label{Curv1}\\
R_{\alpha\beta kl}&=h_{km}^\alpha h_{ml}^\beta-h_{km}^\beta h_{ml}^\alpha,\label{Curv2}\\
R_{\alpha\beta 12k}&=2(b^\beta a_k^\alpha+a^\alpha h_{12k}^\beta-b^\alpha a_k^\beta-a^\beta h_{12k}^\alpha).\label{Curv3}
\end{align}
From now on, we assume that the $2$-dimensional manifold $M$ is minimally immersed in $\mathbb{S}^N(1)$. The Simons identity for minimal submanifolds in a unit sphere is
\begin{equation}\label{SId}
\frac{1}{2}\tri S=\mathcal{B}_1+2S-|A|^2-\rho^\bot.
\end{equation}
We now establish an improved version of Proposition 2.4 of Yang~(cf. \cite{30}). The proof follows the idea there, however, without the assumption on flatness of the normal bundle of $M$.
\begin{thm}\label{NSId}
Suppose that $M$ is a closed surface minimally immersed in a unit sphere $\mathbb{S}^N(1)$ with positive Gaussian curvature. We have
\begin{equation}\label{TWINST}
\frac{1}{2}\tri S=\mathcal{B}_1-\frac{1}{2}S(3S-4).
\end{equation}
\end{thm}
\begin{proof}
Let $\{e_1,\cdots,e_N\}$ be a local orthonormal frame field on $M$ as before, and we have
\begin{equation}\label{Div}
\begin{aligned}
(h_{ijk}^\alpha\tri h_{ij}^\alpha)_k&=\sum_{i,j,\alpha}(\tri h_{ij}^\alpha)^2+h_{ijk}^\alpha h_{pik}^\alpha R_{pljl}+h_{ijk}^\alpha h_{lpk}^\alpha R_{pijl}+h_{ijk}^\alpha h_{pi}^\alpha R_{pljlk}\\
&\hspace{1.3em}+h_{ijk}^\alpha h_{lp}^\alpha R_{pijlk}+h_{ijk}^\alpha h_{lik}^\delta R_{\delta\alpha jl}+h_{ijk}^\alpha h_{li}^\delta R_{\delta\alpha jlk}.\\
\end{aligned}
\end{equation}
By \eqref{Lap1}, \eqref{Curv1} and \eqref{Curv2}, we have
\begin{equation*}
\begin{aligned}
\tri h_{11}^\beta&=a^\beta(2-S)+b^\delta (2a^\delta b^\beta-2a^\beta b^\delta),\\
\tri h_{12}^\beta&=b^\beta(2-S)+a^\delta(2a^\beta b^\delta-2a^\delta b^\beta).
\end{aligned}
\end{equation*}
The first term on the right-hand side of \eqref{Div} becomes
\begin{equation}\label{1st}
\begin{aligned}
\sum_{i,j,\alpha}(\tri h_{ij}^\alpha)^2&=2\sum_{\alpha}(\tri a^\alpha)^2+2\sum_{\alpha}(\tri b^\alpha)^2\\
&=S(2-S)^2+\frac{1}{2}(5S-8)(-S^2+|A|^2+\rho^{\bot}).
\end{aligned}
\end{equation}
By \eqref{Curv1} and \eqref{Curv2}, the second, the third and the sixth terms are
\begin{equation}\label{2nd}
\begin{aligned}
&\hspace{1.3em}h_{ijk}^{\alpha}h_{pik}^{\alpha}R_{pljl}+h_{ijk}^{\alpha}h_{lpk}^{\alpha}R_{pijl}+h_{ijk}^{\alpha}h_{lik}^{\delta}R_{\delta\alpha jl}\\
&=(2-S)\mathcal{B}_1+\sum_{\alpha,\delta}4(a_1^{\alpha}a_2^{\delta}-a_2^{\alpha}a_1^{\delta})R_{\delta\alpha 12}\\
&=(2-S)\mathcal{B}_1+16\l a,a_2\r\l b,a_1\r-16\l a,a_1\r\l b,a_2\r.
\end{aligned}
\end{equation}
By \eqref{Curv1}, the fourth and the fifth terms are equal to
\begin{equation}\label{3rd}
h_{ijk}^{\alpha}h_{pi}^{\alpha}R_{pljlk}+h_{ijk}^{\alpha}h_{lp}^{\alpha}R_{pijlk}=-h_{ij}^{\alpha}h_{ijk}^{\alpha}S_k=-\frac{1}{2}|\nabla S|^2.
\end{equation}
By \eqref{Curv3}, the last term becomes
\begin{equation}\label{4th}
\begin{aligned}
h_{ijk}^{\alpha}h_{li}^{\delta}R_{\delta\alpha jlk}&=2(b^{\delta}a_1^{\alpha}-a^{\delta}a_2^{\alpha})R_{\delta\alpha 121}+2(a^{\delta}a_1^{\alpha}+b^{\delta}a_2^{\alpha})R_{\delta\alpha 122}\\
&=-\frac{1}{2}S\mathcal{B}_1+\frac{1}{4}|\nabla S|^2+16\l a,a_2\r\l b,a_1\r-16\l a,a_1\r\l b,a_2\r.
\end{aligned}
\end{equation}
Combining \eqref{Div}, \eqref{1st}, \eqref{2nd}, \eqref{3rd} and \eqref{4th}, we get
\begin{equation}\label{Diff}
\begin{aligned}
(h_{ijk}^\alpha\tri h_{ij}^\alpha)_k&=\frac{1}{2}(4-3S)\mathcal{B}_1+(2-S)^2S+\frac{1}{2}(5S-8)(-S^2+|A|^2+\rho^\bot)\\
&\hspace{1.3em}-\frac{1}{4}|\nabla S|^2+32\l a,a_2\r\l b,a_1\r-32\l a,a_1\r\l b,a_2\r\\
&=\frac{1}{2}(2-S)S^2+(2-S)(S^2-|A|^2-\rho^{\bot})+\tri S-\frac{3}{8}\tri S^2\\
&\hspace{1.3em}+8(\l a,a_2\r+\l b,a_1\r)^2+8(\l a,a_1\r-\l b,a_2\r)^2,
\end{aligned}
\end{equation}
where in the last equality, we used the equation \eqref{SId} and the relation $$S\tri S=\frac{1}{2}\tri S^2-|\nabla S|^2.$$ Integration over $M$ on both sides of \eqref{Diff} gives
\begin{equation*}
\begin{aligned}
0&=\int_M\left(\frac{1}{2}S^2(2-S)-(2-S)(-S^2+|A|^2+\rho^{\bot})\right)+\int_M8(\l a,a_2\r+\l b,a_1\r)^2+8(\l a,a_1\r-\l b,a_2\r)^2).
\end{aligned}
\end{equation*}
Therefore,
\begin{equation}\label{LIneq}
\int_M(2-S)(-S^2+|A|^2+\rho^{\bot})\geq\int_M\frac{1}{2}(2-S)S^2.
\end{equation}
On the other hand, we have the following relation by direct computation:
\begin{equation}\label{PIneq}
-S^2+|A|^2+\rho^{\bot}\leq\frac{1}{2}S^2,
\end{equation}
where the equality holds if and only if $a\perp b$ and $|a|=|b|$. The positivity of Gaussian curvature implies that $S<2$. Therefore,
\begin{equation}\label{RIneq}
\int_M(2-S)(-S^2+|A|^2+\rho^{\bot})\leq\int_M\frac{1}{2}(2-S)S^2.
\end{equation}
By \eqref{LIneq}, \eqref{PIneq} and \eqref{RIneq}, we get
\begin{equation}\label{equ}
-S^2+|A|^2+\rho^{\bot}=\frac{1}{2}S^2.
\end{equation}
Combining \eqref{SId} and \eqref{equ}, we complete the proof.
\end{proof}

\begin{cor}
Suppose that $M$ is a closed surface minimally immersed in a unit sphere $\mathbb{S}^N(1)$ with positive Gaussian curvature. Under the foregoing assumptions and notations, we have
\begin{align}
\l a,b\r&=0,~|a|^2=|b|^2=\frac{1}{4}S,\label{Funda}\\
\tri a&=\frac{1}{2}a(4-3S),~\tri b=\frac{1}{2}b(4-3S),\label{Twins1}\\
\tri a_1&=\frac{1}{2}a_1(14-9S)+\frac{7}{4}(-aS_1+bS_2),\label{Dual1}\\
\tri a_2&=\frac{1}{2}a_2(14-9S)-\frac{7}{4}(bS_1+aS_2),\label{Dual2}\\ \nonumber
|A|^2&=\frac{1}{2}S^2~\text{and}~\rho^\bot=S^2.
\end{align}
\end{cor}

\begin{proof}
On one hand, \eqref{Funda} is obtained by \eqref{Diff} and \eqref{equ} in the proof of Theorem \ref{NSId}. On the other hand, \eqref{Twins1} can be calculated directly by \eqref{Lap1}, \eqref{Curv1}, \eqref{Curv2} and \eqref{Funda}. Also, \eqref{Dual1} and \eqref{Dual2} can be calculated directly by \eqref{Lap2}, \eqref{Curv1}, \eqref{Curv2}, \eqref{Curv3} and \eqref{Funda} in a similar way. Finally, by definition, we have
\begin{equation*}
\begin{aligned}
|A|^2&=\sum(2a^\alpha a^\beta+2b^\alpha b^\beta)^2\\
&=\sum(4|a|^2|a|^2+4|b|^2|b|^2+8\l a,b\r^2)\\
&=8|a|^4\\
&=\frac{1}{2}S^2
\end{aligned}
\end{equation*}
and
\begin{equation*}
\begin{aligned}
\rho^\bot=\sum|[S_\alpha,S_\beta]|^2&=8\sum(a^\alpha b^\beta-a^\beta b^\alpha)^2\\
&=8\sum(|a|^2|b|^2+|a|^2|b|^2-2\l a,b\r^2)\\
&=16|a|^4\\
&=S^2,
\end{aligned}
\end{equation*}
which prove the corollary.
\end{proof}

\begin{rem}
Theorem \ref{NSId} can also be proved by using the method of holomorphic functions~\emph{(}cf. \emph{\cite{11})}. We will use this method to prove Lemma \ref{FirD} and Proposition \ref{SeconD} in the following.
\end{rem}

\begin{lem}\label{FirD}
Suppose that $M$ is a closed surface minimally immersed in a unit sphere $\mathbb{S}^N(1)$ with positive Gaussian curvature. We have
\begin{equation*}
\l a_1,a_2\r=0\text{ and }|a_1|^2=|a_2|^2=\frac{1}{8}\mathcal{B}_1.
\end{equation*}
\end{lem}

\begin{proof}
Define $$\phi\coloneqq(|a_1|^2-|a_2|^2-2\operatorname{i}\l a_1,a_2\r)dz^6.$$ Then, $\phi$ is a differential form of degree $6$. It can be verified that $\phi$ is independent of the choice of the vector field. We now prove that $\phi$ is actually holomorphic by showing that it satisfies Cauchy-Riemann equations. First we have
\begin{align*}
&\hspace{1.3em}e_1(|a_1|^2-|a_2|^2)+e_2(2\l a_1,a_2\r)\\
&=2(\l a_1,a_{11}\r-\l a_2,a_{21}\r)+2(\l a_{12},a_2\r+\l a_1,a_{22}\r)\\
&=2(\l a_1,\tri a\r-\l a_2,\tri b\r)\\
&=0,
\end{align*} 
where we used \eqref{Cod} and \eqref{Twins1} in the last equation. By the same process, it can be shown that $e_2(|a_1|^2-|a_2|^2)-e_1(2\l a_1,a_2\r)=0$. Therefore, $\phi$ is holomorphic. But the holomorphic differential on a $2$-dimensional sphere must be zero. The result follows.
\end{proof}

\begin{cor}\label{rem3}
Suppose that $M$ is a closed surface minimally immersed in a unit sphere $\mathbb{S}^N(1)$ with positive Gaussian curvature. Under the foregoing assumptions and notations, we have
\begin{enumerate}
\item[\emph{(i)}]
\begin{equation}\label{deriv01}
\begin{aligned}
\l a,a_1\r&=\l b,a_2\r=\frac{1}{8}S_1,\\
\l a,a_2\r&=-\l b,a_1\r=\frac{1}{8}S_2,
\end{aligned}
\end{equation}
\item[\emph{(ii)}]
\begin{equation}\label{deriv02}
\begin{aligned}
\l a,a_{11}\r&=\l b,a_{21}\r=\frac{1}{8}(S_{11}-\mathcal{B}_1),\\
\l a,a_{22}\r&=-\l b,a_{12}\r=\frac{1}{8}(S_{22}-\mathcal{B}_1),\\
\l a,a_{12}\r&=\l b,a_{22}\r=\frac{1}{8}S_{12},\\
\l a,a_{21}\r&=-\l b,a_{11}\r=\frac{1}{8}S_{21},
\end{aligned}
\end{equation}
\item[\emph{(iii)}] and
\begin{equation}\label{deriv12}
\begin{aligned}
\l a_1,a_{21}\r&=-\l a_2,a_{11}\r,\\
\l a_1,a_{22}\r&=-\l a_2,a_{12}\r,\\
\l a_1,a_{11}\r&=\l a_2,a_{21}\r=\frac{1}{16}(\mathcal{B}_1)_1,\\
\l a_1,a_{12}\r&=\l a_2,a_{22}\r=\frac{1}{16}(\mathcal{B}_1)_2.
\end{aligned}
\end{equation}
\end{enumerate}
\end{cor}
\begin{proof}
First, \eqref{deriv01} can be obtained by differentiating \eqref{Funda}. Second, by differentiating \eqref{deriv01} and using Lemma \ref{FirD}, we can obtain \eqref{deriv02}. Finally, \eqref{deriv12} can be obtained by differentiating the equations in Lemma \ref{FirD}.
\end{proof}
Similar to Lemma \ref{FirD}, we  obtain the following proposition.
\begin{prop}\label{SeconD}
Suppose that $M$ is a closed surface minimally immersed in a unit sphere $\mathbb{S}^N(1)$ with positive Gaussian curvature. We have
\begin{enumerate}
\item[\emph{(i)}] $\l a_{11},a_{21}\r=0$ and $$|a_{11}|^2=|a_{21}|^2=\frac{1}{16}\mathcal{B}_2-\frac{1}{32}(3S-4)(S_{11}-S_{22}),$$
\item[\emph{(ii)}] $\l a_{22},a_{12}\r=0$ and $$|a_{22}|^2=|a_{12}|^2=\frac{1}{16}\mathcal{B}_2+\frac{1}{32}(3S-4)(S_{11}-S_{22}).$$
\end{enumerate}
\end{prop}
\begin{proof}
Define $$\phi_1\coloneqq(|a_{11}|^2-|a_{21}|^2-2\operatorname{i}\l a_{11},a_{21}\r)dz^8.$$ Then $\phi_1$ is a differential form of degree $8$. It can be verified that $\phi_1$ is independent of the choice of the vector field. We now prove that $\phi_1$ is actually holomorphic by showing that it satisfies Cauchy-Riemann equations. First we have
\begin{align*}
&\hspace{1.3em}e_1(|a_{11}|^2-|a_{21}|^2)+e_2(2\l a_{11},a_{21}\r)\\
&=2\l a_{11},a_{111}\r+2\l a_{11},a_{212}\r+2\l a_{21},a_{112}\r-2\l a_{21},a_{211}\r\\
&=2\l a_{11},\tri a_1\r-2\l a_{21},\tri a_2\r\\
&=0,
\end{align*}
where we used \eqref{Dual1}, \eqref{Dual2} and the relations
\begin{align*}
a_{212}^\alpha &=a_{122}^\alpha+\frac{3}{2}a_1^\alpha S-\frac{3}{2}b^\alpha S_2-2a_1^\alpha,\\
-a_{112}^\alpha &=a_{222}^\alpha+\frac{3}{2}a_2^\alpha S+\frac{3}{2}a^\alpha S_2-2a_2^\alpha,
\end{align*}
which can be obtained by using $(\ref{Ric1})$. By the same process, it can be shown that $$e_2(|a_{11}|^2-|a_{21}|^2)-e_1(2\l a_{11},a_{21}\r)=-2\l a_{11},\tri a_2\r-2\l a_{21},\tri a_1\r=0.$$ Therefore, $\phi_1$ is holomorphic. Similarly, by defining $$\phi_2\coloneqq(|a_{22}|^2-|a_{12}|^2+2\operatorname{i}\l a_{22},a_{12}\r)dz^8,$$ it can be shown that $\phi_2$ is also holomorphic. But the holomorphic differential on a $2$-dimensional sphere must be zero. Then we obtain $\l a_{11},a_{21}\r=0$, $\l a_{22},a_{12}\r=0$, $|a_{11}|=|a_{21}|$ and $|a_{22}|=|a_{12}|$. By \eqref{Ric1}, we have
\begin{equation*}
\begin{aligned}
a_{12}^\alpha-a_{21}^\alpha&=h_{p1}^\alpha R_{p112}+h_{1p}^\alpha R_{p112}+h_{11}^\beta R_{\beta\alpha 12}\\
&=b^\alpha(S-2)+2a^\beta(a^\beta b^\alpha-a^\alpha b^\beta)\\
&=\frac{1}{2}b^\alpha(3S-4).
\end{aligned}
\end{equation*}
Then, with Theorem \ref{NSId} and \eqref{deriv02} we obtain
\begin{equation*}
\begin{aligned}
|a_{12}|^2&=\l a_{12},a_{12}\r\\
&=\l a_{21}+\frac{1}{2}b(3S-4),a_{21}+\frac{1}{2}b(3S-4)\r\\
&=|a_{21}|^2+2\l a_{21},\frac{1}{2}b(3S-4)\r+\frac{1}{4}\l b(3S-4),b(3S-4)\r\\
&=|a_{21}|^2+\frac{1}{8}(3S-4)(S_{11}-\mathcal{B}_1)+\frac{1}{16}S(3S-4)^2\\
&=|a_{21}|^2+\frac{1}{16}(3S-4)(S_{11}-S_{22}).
\end{aligned}
\end{equation*}
Therefore,
\begin{equation*}
\begin{aligned}
\mathcal{B}_2&=\sum4\left[(h_{1111}^\alpha)^2+(h_{1112}^\alpha)^2+(h_{1121}^\alpha)^2+(h_{1122}^\alpha)^2\right]\\
&=8\left(|a_{12}|^2+|a_{21}|^2\right)\\
&=16|a_{21}|^2+\frac{1}{2}(3S-4)(S_{11}-S_{22})\\
&=16|a_{12}|^2-\frac{1}{2}(3S-4)(S_{11}-S_{22}),
\end{aligned}
\end{equation*}
which completes the proof.
\end{proof}

\section{The Simon conjecture in the cases $s=1,2$}
\subsection{The first Simons-type identity}
\begin{thm}\label{thm gap s=1}
Let $M$ be a closed minimal surface immersed in $\mathbb{S}^N(1)$ with positive Gaussian curvature. Then
\begin{equation*}
\int_MS(3S-4)=2\int_M\mathcal{B}_1\ge0.
\end{equation*}
In particular, if $0\leq S\leq \frac{4}{3}$, then $S=0$ or $S=\frac{4}{3}$.
\end{thm}
\begin{proof}
Integrating on both sides of \eqref{TWINST} and combining $\mathcal{B}_1\ge0$, we obtain
\begin{equation}\label{FirInt}
\frac{1}{2}\int_MS(3S-4)=\int_M\mathcal{B}_1\ge0.
\end{equation}
If $0\le S\le\frac{4}{3}$, then $S(3S-4)\le0$. Combining \eqref{FirInt}, we arrive at $S(3S-4)=0$. It follows that $S=0$ or $S=\frac{4}{3}$.
\end{proof}

\subsection{The second Simons-type identity}
\begin{thm}\label{thm gap s=2}
Let $M$ be a closed minimal surface immersed in $\mathbb{S}^N(1)$ with positive Gaussian curvature. Then
\begin{equation*}
\int_MS(3S-4)(3S-5)=2\int_M[\mathcal{B}_2-\frac{1}{4}S(3S-4)^2+\frac{1}{2}|\nnn S|^2]\ge0.
\end{equation*}
In particular, if $\frac{4}{3}\leq S\leq\frac{5}{3},$ then $S=\frac{4}{3}$ or $S=\frac{5}{3}.$
\end{thm}
\begin{proof}
By \eqref{Lap2}, We have
\begin{equation}\label{Sum1}
\begin{aligned}
h_{ijk}^\alpha\tri h_{ijk}^\alpha&=(h_{ijk}^\alpha\tri h_{ij}^\alpha)_k-\sum(\tri h_{ij}^\alpha)^2+2h_{ijk}^\alpha h_{pj}^\alpha R_{pikmm}+h_{ijk}^\alpha h_{ijp}^\alpha R_{pmkm}\\
&\hspace{1.3em}+4h_{ijk}^\alpha h_{pjm}^\alpha R_{pikm}+2h_{ijk}^\alpha h_{ijm}^\beta R_{\beta\alpha km}+h_{ijk}^\alpha h_{ij}^\beta R_{\beta\alpha kmm}.
\end{aligned}
\end{equation} 
By \eqref{Twins1}, we get
\begin{equation}\label{sumlapS}
\sum_\alpha(\tri h_{ij}^\alpha)^2=2\sum_\alpha[(\tri a^\alpha)^2+(\tri b^\alpha)^2]=\frac{1}{4}S(3S-4)^2.
\end{equation} 
By \eqref{Curv1}, it follows that
\begin{equation}\label{sumpinch2term3}
\begin{aligned}
2h_{ijk}^\alpha h_{pj}^\alpha R_{pikmm}&=-h_{ijk}^\alpha h_{pj}^\alpha(\delta_{pk}\delta_{im}-\delta_{pm}\delta_{ik})S_m\\
&=-h_{ijm}^\alpha h_{ij}^\alpha S_m\\
&=-\frac{1}{2}|\nabla S|^2,
\end{aligned}
\end{equation} 
and
\begin{equation}\label{sumpinch2term45}
4h_{ijk}^\alpha h_{pjm}^\alpha R_{pikm}+h_{ijk}^\alpha h_{ijp}^\alpha R_{pmkm}=5(1-\frac{S}{2})\mathcal{B}_1.
\end{equation} 
By \eqref{Curv2}, we obtain
\begin{equation}\label{sumpinch2term6}
2h_{ijk}^\alpha h_{ijm}^\beta R_{\beta\alpha km}=16(a_1^\alpha a_2^\beta-a_2^\alpha a_1^\beta)(a^\beta b^\alpha-a^\alpha b^\beta)=-\frac{1}{2}|\nabla S|^2.
\end{equation} 
By \eqref{Curv3}, we get
\begin{equation}\label{Last}
\begin{aligned}
h_{ijk}^\alpha h_{ij}^\beta R_{\beta\alpha kmm}&=h_{ij1}^\alpha h_{ij}^\beta R_{\beta\alpha 122}-h_{ij2}^\alpha h_{ij}^\beta R_{\beta\alpha 121}\\
&=4(a_1^\alpha a^\beta+a_2^\alpha b^\beta)(a_2^\beta b^\alpha-a^\beta a_1^\alpha-a_2^\alpha b^\delta+a^\alpha a_1^\beta)\\
&\hspace{1.3em}-4(a_2^\alpha a^\beta-a_1^\alpha b^\beta)(a_1^\beta b^\alpha+a^\beta a_2^\alpha-a_1^\alpha b^\beta-a^\alpha a_2^\beta)\\
&=-\frac{1}{2}S\mathcal{B}_1.
\end{aligned}
\end{equation} 
Combining \eqref{Sum1}, \eqref{sumlapS}, \eqref{sumpinch2term3}, \eqref{sumpinch2term45}, \eqref{sumpinch2term6} and \eqref{Last}, we obtain
\begin{equation}\label{SecSum}
\begin{aligned}
h_{ijk}^\alpha\tri h_{ijk}^\alpha=(h_{ijk}^\alpha\tri h_{ij}^\alpha)_k+\frac{5}{2}\tri S-\frac{3}{4}\tri S^2+\frac{1}{2}|\nabla S|^2-\frac{1}{4}S(3S-4)(9S-14).
\end{aligned}
\end{equation} 
Also, with \eqref{sumlapS} we can give a representation of $\mathcal{B}_2$ as
\begin{equation}\label{B2B}
\begin{aligned}
\mathcal{B}_2&=4\sum\left[(a_{11}^\alpha)^2+(a_{22}^\alpha)^2+(a_{12}^\alpha)^2+(a_{21}^\alpha)^2\right]\\
&=2\sum\left[(a_{11}^\alpha+a_{22}^\alpha)^2+(a_{12}^\alpha-a_{21}^\alpha)^2\right]+2\sum\left[(a_{11}^\alpha-a_{22}^\alpha)^2+(a_{12}^\alpha+a_{21}^\alpha)^2\right]\\
&=2\sum\left[(\tri a^\alpha)^2+(\tri b^\alpha)^2\right]+2\sum\left[(a_{11}^\alpha-a_{22}^\alpha)^2+(a_{12}^\alpha+a_{21}^\alpha)^2\right]\\
&=\frac{1}{4}S(3S-4)^2+\mathcal{C}_1,
\end{aligned}
\end{equation}
where
\begin{equation*}
\mathcal{C}_1=2\sum\left[(a_{11}^\alpha-a_{22}^\alpha)^2+(a_{12}^\alpha+a_{21}^\alpha)^2\right]=\mathcal{B}_2-\frac{1}{4}S(3S-4)^2\ge0.
\end{equation*}
By \eqref{SecSum} and \eqref{B2B}, we conclude that
\begin{equation}\label{LapFir}
\begin{aligned}
\frac{1}{2}\tri\mathcal{B}_1&=h_{ijk}^\alpha\tri h_{ijk}^\alpha+\mathcal{B}_2\\
&=(h_{ijk}^\alpha\tri h_{ij}^\alpha)_k+\frac{5}{2}\tri S-\frac{3}{4}\tri S^2+\frac{1}{2}|\nnn S|^2-\frac{1}{2}S(3S-4)(3S-5)+\mathcal{C}_1.
\end{aligned}
\end{equation} 
Integrating on both sides of \eqref{LapFir}, we get
\begin{equation}\label{FirInts=2}
\begin{aligned}
\frac{1}{2}\int_MS(3S-4)(3S-5)&=\int_M(\mathcal{C}_1+\frac{1}{2}|\nnn S|^2)\\
&=\int_M[\mathcal{B}_2-\frac{1}{4}S(3S-4)^2+\frac{1}{2}|\nnn S|^2]\\
&\ge0.
\end{aligned}
\end{equation} 
If $\frac{4}{3}\le S\le\frac{5}{3}$, then $S(3S-4)(3S-5)\le0$. Combining \eqref{FirInts=2}, we arrive at $S(3S-4)(3S-5)=0$. It follows that $S=\frac{4}{3}$ or $S=\frac{5}{3}$.
\end{proof}

\begin{rem}
By \eqref{Diff}, \eqref{equ}, \eqref{Funda} and \eqref{SecSum}, we conclude that
\begin{equation}\label{SecEq}
\frac{1}{2}\tri\mathcal{B}_1=\frac{7}{2}\tri S-\frac{9}{8}\tri S^2+\frac{1}{2}|\nnn S|^2-\frac{1}{4}S(3S-4)(9S-14)+\mathcal{B}_2.
\end{equation}
\end{rem}

\section{The case $s=3$}
We now consider the case $s=3$.
\subsection{The lower bound of $\mathcal{B}_3$}
\begin{lem}\label{lem a lower bound to B3}
Suppose that $M$ is a closed surface minimally immersed in a unit sphere $\mathbb{S}^N(1)$ with positive Gaussian curvature. Under the foregoing assumptions and notations, we have
\begin{equation*}
\mathcal{B}_3=\sum(h_{ijklm}^\alpha)^2=\frac{1}{4}(45S^2-144S+116)\mathcal{B}_1+\frac{13}{8}(7S-8)|\nabla S|^2+\mathcal{C}_2+\mathcal{C}_3,
\end{equation*}
where $\mathcal{C}_2=2\sum\left[(a_{111}^\alpha-a_{122}^\alpha)^2+(a_{211}^\alpha-a_{222}^\alpha)^2\right]$ and $\mathcal{C}_3=2\sum\left[(a_{112}^\alpha+a_{121}^\alpha)^2+(a_{212}^\alpha+a_{221}^\alpha)^2\right]$.
\end{lem}

\begin{proof}
By \eqref{Dual1}, \eqref{Dual2} and Corollary \ref{rem3}, we have
\begin{equation*}\label{sumb1}
\begin{aligned}
&\hspace{1.1em}\sum\left[(\tri a_1^\alpha)^2+(\tri a_2^\alpha)^2\right]\\
&=\sum\left[\left(\frac{1}{2}a_1^\alpha(14-9S)+\frac{7}{4}(-a^\alpha S_1+b^\alpha S_2)\right)^2\right.\left.+\left(\frac{1}{2}a_1^\alpha(14-9S)+\frac{7}{4}(-a^\alpha S_1+b^\alpha S_2)\right)^2\right]\\
&=\frac{1}{16}(9S-14)^2\mathcal{B}_1+\frac{7}{32}(25S-28)|\nnn S|^2.
\end{aligned}
\end{equation*}
Then we obtain
\begin{equation}\label{Ineq1}
\begin{aligned}
&\hspace{1.3em}4\sum\left[(a_{111}^\alpha)^2+(a_{122}^\alpha)^2+(a_{211}^\alpha)^2+(a_{222}^\alpha)^2\right]\\
&=2\sum\left[(a_{111}^\alpha+a_{122}^\alpha)^2+(a_{211}^\alpha+a_{222}^\alpha)^2\right]+2\sum\left[(a_{111}^\alpha-a_{122}^\alpha)^2+(a_{211}^\alpha-a_{222}^\alpha)^2\right]\\
&=2\sum\left[(\tri a_1^\alpha)^2+(\tri a_2^\alpha)^2\right]+\mathcal{C}_2\\
&=\frac{1}{8}(9S-14)^2\mathcal{B}_1+\frac{7}{16}(25S-28)|\nabla S|^2+\mathcal{C}_2.
\end{aligned}
\end{equation}
By \eqref{Ric2} and \eqref{Curv2}, we have
\begin{equation*}\label{Ineq2}
\begin{aligned}
&\hspace{1.3em}4\sum\left[(a_{112}^\alpha)^2+(a_{121}^\alpha)^2+(a_{212}^\alpha)^2+(a_{221}^\alpha)^2\right]\\
&=2\sum\left[(a_{112}^\alpha-a_{121}^\alpha)^2+(a_{212}^\alpha-a_{221}^\alpha)^2\right]+2\sum\left[(a_{112}^\alpha+a_{121}^\alpha)^2+(a_{212}^\alpha+a_{221}^\alpha)^2\right]\\
&=2\left[(3h_{p11}^\alpha R_{p112}+h_{111}^\beta R_{\beta\alpha12})^2+(3h_{p12}^\alpha R_{p112}+h_{112}^\beta R_{\beta\alpha 12})^2\right]+\mathcal{C}_3\\
&=2\left[(3h_{211}^\alpha R_{2112}+h_{111}^\beta R_{\beta\alpha12})^2+(3h_{212}^\alpha R_{2112}+h_{112}^\beta R_{\beta\alpha 12})^2\right]+\mathcal{C}_3\\
&=2\left[(3a_2^\alpha(\frac{S}{2}-1)+a_1^\beta R_{\beta\alpha12})^2+(3a_1^\alpha(1-\frac{S}{2})+a_2^\beta R_{\beta\alpha 12})^2\right]+\mathcal{C}_3\\
&=18(1-\frac{S}{2})^2(|a_1|^2+|a_2|^2)+2(a_1^\beta R_{\beta\alpha 12})^2+2(a_2^\beta R_{\beta\alpha 12})^2+24a_1^\alpha a_2^\beta(1-\frac{S}{2})R_{\beta\alpha 12}+\mathcal{C}_3\\
&=18(1-\frac{S}{2})^2\frac{1}{4}\mathcal{B}_1+\frac{1}{32}S|\nnn S|^2+\frac{1}{32}S|\nnn S|^2-\frac{3}{4}(1-\frac{S}{2})|\nnn S|^2+\mathcal{C}_3\\
&=\frac{9}{2}(1-\frac{S}{2})^2\mathcal{B}_1+\frac{1}{16}(7S-12)|\nabla S|^2+\mathcal{C}_3.
\end{aligned}
\end{equation*}
Combining these two parts gives that
\begin{equation*}\label{LowB}
\begin{aligned}
\mathcal{B}_3
&=\sum(h_{ijklm}^\alpha)^2\\
&=4\sum\left[(a_{111}^\alpha)^2+(a_{122}^\alpha)^2+(a_{211}^\alpha)^2+(a_{222}^\alpha)^2+(a_{112}^\alpha)^2+(a_{121}^\alpha)^2+(a_{212}^\alpha)^2+(a_{221}^\alpha)^2\right]\\
&=\frac{1}{8}\mathcal{B}_1(9S-14)^2+\frac{7}{16}|\nabla S|^2(25S-28)+\frac{9}{2}\mathcal{B}_1(1-\frac{S}{2})^2+\frac{1}{16}(7S-12)|\nabla S|^2+\mathcal{C}_2+\mathcal{C}_3\\
&=\frac{1}{4}(45S^2-144S+116)\mathcal{B}_1+\frac{13}{8}(7S-8)|\nabla S|^2+\mathcal{C}_2+\mathcal{C}_3,
\end{aligned}
\end{equation*}
which completes the proof.
\end{proof}

\subsection{Laplacian of $\mathcal{B}_2$}
\begin{lem}
Let $M$ be a closed minimal surface immersed in $\mathbb{S}^N(1)$ with positive Gaussian curvature. Then we have
\begin{align}
h_{ijkl}^\alpha h_{ijlk}^\alpha&=\mathcal{B}_2-\frac{1}{4}S(3S-4)^2,\label{Swap} \\
h_{ijkk}^\alpha h_{ijll}^\alpha&=\frac{1}{4}S(3S-4)^2.\label{Double}
\end{align}
\end{lem}
\begin{proof}
The second identity can be obtained by \eqref{sumlapS} directly. For the first identity, by using \eqref{Ric1} we obtain
\begin{align*}
h_{ijkl}^\alpha h_{ijlk}^\alpha &=h_{ijkl}^\alpha h_{ijkl}^\alpha+h_{ijkl}^\alpha(2h_{ip}^\alpha R_{pjlk}+h_{ij}^\beta R_{\beta\alpha lk})\\
&=\mathcal{B}_2+2(h_{ij12}^\alpha-h_{ij21}^\alpha)h_{ip}^\alpha R_{pj21}+h_{ij}^\beta(h_{ij21}^\alpha-h_{ij12}^\alpha)R_{\beta\alpha 12}\\
&=\mathcal{B}_2+2(2h_{iq}^\alpha R_{qj12}+h_{ij}^\beta R_{\beta\alpha 12})h_{ip}^\alpha R_{pj21}+h_{ij}^\beta(2h_{ip}^\alpha R_{pj21}+h_{ij}^\gamma R_{\gamma\alpha 21})R_{\gamma\alpha12}\\
&=\mathcal{B}_2+S(1-\frac{S}{2})(3S-4)-\frac{1}{4}S^2(3S-4)\\
&=\mathcal{B}_2-\frac{1}{4}S(3S-4)^2,
\end{align*}
which completes the proof.
\end{proof}
We now derive the Laplacian of $\mathcal{B}_2$.
\begin{thm}\label{thm laplacian B2}
Let $M$ be a closed minimal surface immersed in $\mathbb{S}^N(1)$ with positive Gaussian curvature. Then
\begin{equation*}
\begin{aligned}
\frac{1}{2}\tri\mathcal{B}_2&=(h^\alpha_{ijkl}\tri h^\alpha_{ijk})_l-(21S^2-64S+49)\mathcal{B}_1+7(1-\frac{S}{2})\mathcal{B}_2+\frac{1}{4}S(3S-4)^2(7S-12)\\
&\hspace{1.4em}-\frac{7}{2}(7S-8)|\nnn S|^2-\l\nnn\mathcal{B}_1,\nnn S\r+\frac{1}{4}(\tri S)^2-\frac{1}{2}|\text{\rm Hess}~S|^2+\mathcal{B}_3.
\end{aligned}
\end{equation*}
\end{thm}

\begin{proof}
By \eqref{Ric3}, \eqref{Lap1} and \eqref{Lap2}, we get
\begin{equation*}
\begin{aligned}
\tri h^\alpha_{ijkl}
&=h^\alpha_{ijklmm}\\
&=(h^\alpha_{ijkml}+h^\alpha_{pjk}R_{pilm}+h^\alpha_{ipk}R_{pjlm}+h^\alpha_{ijp}R_{pklm}+h^\beta_{ijk}R_{\beta\alpha lm})_m\\
&=h^\alpha_{ijkmlm}+(h^\alpha_{pjk}R_{pilm}+h^\alpha_{ipk}R_{pjlm}+h^\alpha_{ijp}R_{pklm}+h^\beta_{ijk}R_{\beta\alpha lm})_m\\
&=h^\alpha_{ijkmml}+h^\alpha_{pjkm}R_{pilm}+h^\alpha_{ipkm}R_{pjlm}+h^\alpha_{ijpm}R_{pklm}+h^\alpha_{ijkp}R_{pmlm}+h^\beta_{ijkm}R_{\beta\alpha lm}\\
&\hspace{1.3em}+(h^\alpha_{pjk}R_{pilm}+h^\alpha_{ipk}R_{pjlm}+h^\alpha_{ijp}R_{pklm}+h^\beta_{ijk}R_{\beta\alpha lm})_m\\
&=(\triangle h^\alpha_{ijk})_l+2h^\alpha_{pjkm}R_{pilm}+2h^\alpha_{ipkm}R_{pjlm}+2h^\alpha_{ijpm}R_{pklm}+h^\alpha_{ijkp}R_{pmlm}\\
&\hspace{1.3em}+h^\alpha_{pjk}R_{pilmm}+h^\alpha_{ipk}R_{pjlmm}+h^\alpha_{ijp}R_{pklmm}+h^\beta_{ijk}R_{\beta\alpha lmm}+2h^\beta_{ijkm}R_{\beta\alpha lm}.
\end{aligned}
\end{equation*} 
It follows that
\begin{equation}\label{FirEq}
\begin{aligned}
h^\alpha_{ijkl}\triangle h^\alpha_{ijkl}&=h^\alpha_{ijkl}(\triangle h^\alpha_{ijk})_l+6h^\alpha_{ijkl}h^\alpha_{pjkm}R_{pilm}+h^\alpha_{ijkl}h^\alpha_{ijkp}R_{pmlm}\\
&\hspace{1.3em}+3h^\alpha_{ijkl}h^\alpha_{pjk}R_{pilmm}+2h^\alpha_{ijkl}h^\delta_{ijkm}R_{\delta\alpha lm}+h^\alpha_{ijkl}h^\delta_{ijk}R_{\delta\alpha lmm}\\
&=(h^\alpha_{ijkl}\triangle h^\alpha_{ijk})_l-\sum(\triangle h^\alpha_{ijk})^2+6h^\alpha_{ijkl}h^\alpha_{pjkm}R_{pilm}+h^\alpha_{ijkl}h^\alpha_{ijkp}R_{pmlm}\\
&\hspace{1.3em}+3h^\alpha_{ijkl}h^\alpha_{pjk}R_{pilmm}+h^\alpha_{ijkl}h^\delta_{ijk}R_{\delta\alpha lmm}+2h^\alpha_{ijkl}h^\delta_{ijkm}R_{\delta\alpha lm}.
\end{aligned}
\end{equation} 
By \eqref{sumb1}, the second term is equal to
\begin{equation}\label{2ndterm}
\begin{aligned}
-\sum(\tri h_{ijk}^\alpha)^2&=-4\sum(\tri a_1^\alpha)^2+(\tri a_2^\alpha)^2\\
&=-\frac{1}{4}(9S-14)^2\mathcal{B}_1-\frac{7}{8}(25S-28)|\nabla S|^2.
\end{aligned}
\end{equation} 
By \eqref{Swap}, \eqref{Double} and the minimality, the third term and the fourth term are
\begin{equation}\label{3rd4thterm}
\begin{aligned}
6h^\alpha_{ijkl}h^\alpha_{pjkm}R_{pilm}+h^\alpha_{ijkl}h^\alpha_{ijkp}R_{pmlm}=7(1-\frac{S}{2})\mathcal{B}_2-3S(1-\frac{S}{2})(3S-4)^2.
\end{aligned}
\end{equation} 
By Lemma \ref{FirD}, the fifth term is equal to
\begin{equation}\label{5thterm}
\begin{aligned}
3h_{ijkl}^\alpha h_{pjk}^\alpha R_{pilmm}&=3h_{ijkl}^\alpha h_{ljk}^\alpha R_{lilii}+3h_{ijki}^\alpha h_{mjk}^\alpha R_{miimm}\\
&=-12(h_{111}^\alpha h_{1121}^\alpha S_2+h_{112}^\alpha h_{1112}^\alpha S_1)\\
&=-\frac{3}{4}((\mathcal{B}_1)_2+S_2(3S-4))S_2-\frac{3}{4}((\mathcal{B}_1)_1+S_1(3S-4))S_1\\
&=-\frac{3}{4}\l\nabla\mathcal{B}_1,\nabla S\r-\frac{3}{4}(3S-4)|\nabla S|^2.
\end{aligned}
\end{equation} 
By Lemma \ref{FirD} and \eqref{Ric1}, the sixth term is
\begin{equation}\label{Final}
\begin{aligned}
h_{ijkl}^\alpha h_{ijk}^\delta R_{\delta\alpha lmm}&=h_{ijkl}^\alpha h_{ijk}^\delta R_{\delta\alpha122}+h_{ijk2}^\alpha h_{ijk}^\delta R_{\delta\alpha211}\\
&=2h_{ijk}^\delta(h_{ijk1}^\alpha(h_{12}^\alpha h_{112}^\delta-h_{111}^\alpha h_{11}^\delta-h_{112}^\alpha h_{12}^\delta+h_{11}^\alpha h_{111}^\delta)\\
&\hspace{1.3em}-h_{ijk2}^\alpha(h_{12}^\alpha h_{111}^\delta+h_{112}^\alpha h_{11}^\delta-h_{111}^\alpha h_{12}^\delta-h_{11}^\alpha h_{112}^\delta))\\
&=-\frac{1}{4}\l\nabla\mathcal{B}_1,\nabla S\r-\frac{1}{8}(3S-4)|\nabla S|^2-\frac{1}{4}S(3S-4)\mathcal{B}_1.
\end{aligned}
\end{equation} 
The final term is
\begin{equation}\label{Finall}
\begin{aligned}
2h^\alpha_{ijkl}h^\delta_{ijkm}R_{\delta\alpha lm}&=64(\l a,a_{11}\r\l a,a_{22}\r+\l b,a_{11}\r\l b,a_{22}\r)\\
&=(S_{11}-\mathcal{B}_1)(S_{22}-\mathcal{B}_1)-(S_{12})^2\\
&=(\mathcal{B}_1)^2-\frac{1}{2}S(3S-4)\triangle S-\frac{1}{2}|\text{\rm Hess}~S|^2\\
&=\frac{1}{4}S^2(3S-4)^2+\frac{1}{4}(\triangle S)^2-\frac{1}{2}|\text{\rm Hess}~S|^2.
\end{aligned}
\end{equation} 
Combining \eqref{FirEq}, \eqref{2ndterm}, \eqref{3rd4thterm}, \eqref{5thterm}, \eqref{Final} and \eqref{Finall} gives that
\begin{equation*}
\begin{aligned}
h^\alpha_{ijkl}\tri h^\alpha_{ijkl}
&=(h^\alpha_{ijkl}\tri h^\alpha_{ijk})_l-\frac{1}{4}(9S-14)^2\mathcal{B}_1-\frac{7}{8}(25S-28)|\nnn S|^2+7(1-\frac{S}{2})\mathcal{B}_2\\
&\hspace{1.3em}-3S(1-\frac{S}{2})(3S-4)^2-\frac{3}{4}\l\nnn\mathcal{B}_1,\nnn S\r-\frac{3}{4}(3S-4)|\nnn S|^2-\frac{1}{4}\l\nnn\mathcal{B}_1,\nnn S\r\\
&\hspace{1.3em}-\frac{1}{8}(3S-4)|\nnn S|^2-\frac{1}{4}S(3S-4)\mathcal{B}_1+\frac{1}{4}S^2(3S-4)^2+\frac{1}{4}(\tri S)^2-\frac{1}{2}|\text{Hess}~S|^2\\
&=(h^\alpha_{ijkl}\tri h^\alpha_{ijk})_l-(21S^2-64S+49)\mathcal{B}_1+7(1-\frac{S}{2})\mathcal{B}_2+\frac{1}{4}S(3S-4)^2(7S-12)\\
&\hspace{1.3em}-\frac{7}{2}(7S-8)|\nnn S|^2-\l\nnn\mathcal{B}_1,\nnn S\r+\frac{1}{4}(\tri S)^2-\frac{1}{2}|\text{Hess}~S|^2.
\end{aligned}
\end{equation*}
Then, we obtain
\begin{equation*}
\begin{aligned}
\frac{1}{2}\tri\mathcal{B}_2&=h^\alpha_{ijkl}\tri h^\alpha_{ijkl}+\mathcal{B}_3\\
&=(h^\alpha_{ijkl}\tri h^\alpha_{ijk})_l-(21S^2-64S+49)\mathcal{B}_1+7(1-\frac{S}{2})\mathcal{B}_2+\frac{1}{4}S(3S-4)^2(7S-12)\\
&\hspace{1.3em}-\frac{7}{2}(7S-8)|\nnn S|^2-\l\nnn\mathcal{B}_1,\nnn S\r+\frac{1}{4}(\tri S)^2-\frac{1}{2}|\text{Hess}~S|^2+\mathcal{B}_3,
\end{aligned}
\end{equation*}
which completes the proof.
\end{proof}

\subsection{The third Simons-type identity}

The following integral equality is the key to proving Theorem \ref{maintheorem} (iii).

\begin{thm}\label{thm An integral inequality}
Let $M$ be a closed minimal surface immersed in $\mathbb{S}^N(1)$ with positive Gaussian curvature. Then
\begin{equation*}
\begin{aligned}
\int_M S(3S-4)(3S-5)(5S-9)
&=\int_M\left[\frac{3}{2}(11S-21)|\nnn S|^2-\frac{5}{4}(\tri S)^2+2\mathcal{C}_2+2\mathcal{C}_3\right]\\
&=2\int_M\left[\mathcal{B}_3-\frac{1}{8}S(3S-4)(45S^2-144S+116)\right.\\
&\hspace{4.5em}\left.+\frac{1}{8}(65S-166)|\nnn S|^2-\frac{5}{8}(\tri S)^2\right],
\end{aligned}
\end{equation*}
where $\mathcal{C}_2=2\sum\left[(a_{111}^\alpha-a_{122}^\alpha)^2+(a_{211}^\alpha-a_{222}^\alpha)^2\right]$ and $\mathcal{C}_3=2\sum\left[(a_{112}^\alpha+a_{121}^\alpha)^2+(a_{212}^\alpha+a_{221}^\alpha)^2\right]$.
\end{thm}
\begin{proof}
By $\eqref{TWINST}$ we obtain
\begin{equation}
\frac{1}{2}\tri\mathcal{B}_1=\frac{1}{4}\tri^2S+\frac{3}{4}\tri S^2-\tri S.
\end{equation}
Combining $\eqref{SecEq}$ we obtain
\begin{equation}\label{thm5eq1}
\mathcal{B}_2=\frac{1}{4}S(3S-4)(9S-14)-\frac{1}{2}|\nnn S|^2+\frac{15}{8}\tri S^2-\frac{9}{2}\tri S+\frac{1}{4}\tri^2S.
\end{equation}
First, by \eqref{TWINST} and Lemma \ref{lem a lower bound to B3}, we obtain
\begin{equation}\label{L1}
\begin{aligned}
&\hspace{1.1em}-(21S^2-64S+49)\mathcal{B}_1+\mathcal{B}_3\\
&=\frac{1}{4}(-39S^2+112S-80)\mathcal{B}_1+\frac{13}{8}(7S-8)|\nnn S|^2+\mathcal{C}_2+\mathcal{C}_3\\
&=\frac{1}{8}S(3S-4)(-39S^2+112S-80)+\frac{13}{8}(7S-8)|\nnn S|^2\\
&\hspace{1.1em}-\frac{39}{8}S^2\tri S+14S\tri S-10\tri S+\mathcal{C}_2+\mathcal{C}_3.
\end{aligned}
\end{equation}
Second, by \eqref{thm5eq1} we obtain
\begin{equation}\label{L2}
\begin{aligned}
7(1-\frac{S}{2})\mathcal{B}_2&=\frac{7}{8}S(2-S)(3S-4)(9S-14)+\frac{7}{4}(S-2)|\nnn S|^2+\frac{105}{8}\tri S^2\\
&\hspace{1.1em}-\frac{105}{16}S\tri S^2-\frac{63}{2}\tri S+\frac{63}{4}S\tri S+\frac{7}{4}\tri^2S-\frac{7}{8}S\tri^2S.
\end{aligned}
\end{equation}
Also, trivial calculations give that
\begin{equation}\label{finalpoly}
\begin{aligned}
&\hspace{1.1em}\frac{1}{8}S(3S-4)(-39S^2+112S-80)+\frac{7}{8}S(2-S)(3S-4)(9S-14)\\
&\hspace{2em}+\frac{1}{4}S(3S-4)^2(7S-12)\\
&=-\frac{1}{2}S(3S-4)(3S-5)(5S-9).
\end{aligned}
\end{equation}
Therefore, combining Theorem {\ref{thm laplacian B2}}, \eqref{L1}, \eqref{L2} and \eqref{finalpoly}, we obtain
\begin{equation}
\begin{aligned}
\frac{1}{2}\tri\mathcal{B}_2&=(h^\alpha_{ijkl}\tri h^\alpha_{ijk})_l-\frac{1}{2}S(3S-4)(3S-5)(5S-9)-\frac{1}{8}(91S-92)|\nnn S|^2\\
&\hspace{1.3em}-\frac{39}{8}S^2\tri S+\frac{119}{4}S\tri S+\frac{105}{8}\tri S^2-\frac{105}{16}S\tri S^2-\frac{83}{2}\tri S\\
&\hspace{1.3em}+\frac{7}{4}\tri^2S-\frac{7}{8}S\tri^2S-\l\nnn \mathcal{B}_1,\nnn S\r+\frac{1}{4}(\tri  S)^2-\frac{1}{2}|\text{\rm Hess}~S|^2+\mathcal{C}_2+\mathcal{C}_3.
\end{aligned}
\end{equation}
To continue, we first have
\begin{align*}
S\tri S&=\frac{1}{2}\tri S^2-|\nnn S|^2,\\
S\tri S^2&=\frac{2}{3}\tri S^3-2S|\nnn S|^2,
\end{align*}
and
\begin{equation*}
S^2\tri S=\frac{1}{3}\tri S^3-2S|\nnn S|^2,
\end{equation*}
which can be obtained by easy calculations. Next, since we have
\begin{equation*}
\text{div}(\tri S\nnn S)=(\tri S)^2+\l\nnn S,\nnn(\tri S)\r
\end{equation*}
and
\begin{equation*}
\tri(S\tri S)=(\tri S)^2+2\l\nnn S,\nnn(\tri S)\r+S\tri^2S,
\end{equation*}
we obtain
\begin{equation*}
S\tri^2S=\tri(S\tri S)-2\text{div}(\tri S\nnn S)+(\tri S)^2.
\end{equation*}
Also, since we have
\begin{equation*}
\mathcal{B}_1\tri S=\frac{1}{2}(\tri S)^2+\frac{3}{2}S^2\tri S-2S\tri S
\end{equation*}
and
\begin{equation*}
\text{div}(\mathcal{B}_1\nnn S)=\mathcal{B}_1\tri S+\l\nnn \mathcal{B}_1,\nnn S\r,
\end{equation*}
we obtain
\begin{equation*}
-\l\nnn \mathcal{B}_1,\nnn S\r=-\text{div}(\mathcal{B}_1\nnn S)+\frac{1}{2}(\tri S)^2+\frac{3}{2}S^2\tri S-2S\tri S.
\end{equation*}
Hence, we have
\begin{equation*}
\begin{aligned}
-\frac{39}{8}S^2\tri S&=-\frac{13}{8}\tri S^3+\frac{39}{4}S|\nnn S|^2,\\
\frac{119}{4}S\tri S&=\frac{119}{8}\tri S^2-\frac{119}{4}|\nnn S|^2,\\
-\frac{105}{16}S\tri S^2&=-\frac{35}{8}\tri S^3+\frac{105}{8}S|\nnn S|^2,\\
-\frac{7}{8}S\tri^2S&=-\frac{7}{8}\tri(S\tri S)+\frac{7}{4}\text{div}(\tri S\nnn S)-\frac{7}{8}(\tri S)^2,~\text{and}\\
-\l\nnn \mathcal{B}_1,\nnn S\r&=-\text{div}(\mathcal{B}_1\nnn S)+\frac{1}{2}(\tri S)^2+\frac{1}{2}\tri S^3-3S|\nnn S|^2-\tri S^2+2|\nnn S|^2.
\end{aligned}
\end{equation*}
Therefore,
\begin{equation*}
\begin{aligned}
\frac{1}{2}\tri\mathcal{B}_2&=(h^\alpha_{ijkl}\tri h^\alpha_{ijk})_l-\frac{1}{2}S(3S-4)(3S-5)(5S-9)-\frac{1}{8}(91S-92)|\nnn S|^2\\
&\hspace{1.3em}-\frac{39}{8}S^2\tri S+\frac{119}{4}S\tri S+\frac{105}{8}\tri S^2-\frac{105}{16}S\tri S^2-\frac{83}{2}\tri S\\
&\hspace{1.3em}+\frac{7}{4}\tri^2S-\frac{7}{8}S\tri^2S-\l\nnn \mathcal{B}_1,\nnn S\r+\frac{1}{4}(\tri  S)^2-\frac{1}{2}|\text{\rm Hess}~S|^2+\mathcal{C}_2+\mathcal{C}_3\\
&=(h^\alpha_{ijkl}\tri h^\alpha_{ijk})_l-\frac{1}{2}S(3S-4)(3S-5)(5S-9)-\frac{11}{2}\tri S^3+27\tri S^2-\frac{83}{2}\tri S\\
&\hspace{1.3em}+\frac{7}{4}\tri^2S-\frac{7}{8}\tri(S\tri S)+\frac{7}{4}\text{div}(\tri S\nnn S)-\text{div}(\mathcal{B}_1\nnn S)+\frac{1}{4}(34S-65)|\nnn S|^2\\
&\hspace{1.3em}-\frac{1}{8}(\tri S)^2-\frac{1}{2}|\text{\rm Hess}~S|^2+\mathcal{C}_2+\mathcal{C}_3.
\end{aligned}
\end{equation*}
Integrating on both sides, we obtain
\begin{equation}\label{beforereilley}
\begin{aligned}
&\hspace{1.3em}\int_M\frac{1}{2}S(3S-4)(3S-5)(5S-9)\\
&=\int_M[\frac{1}{4}(34S-65)|\nnn S|^2-\frac{1}{8}(\tri S)^2-\frac{1}{2}|\text{\rm Hess}~S|^2+\mathcal{C}_2+\mathcal{C}_3].
\end{aligned}
\end{equation}
By Reilly's formula~(cf. \cite{22}), we obtain
\begin{equation*}
\int_M\left[2(\tri S)^2-2|\text{\rm Hess}~S|^2+(S-2)|\nnn S|^2\right]=0.
\end{equation*}
Therefore, $\eqref{beforereilley}$ becomes
\begin{equation*}
\int_M\frac{1}{2}S(3S-4)(3S-5)(5S-9)=\int_M\left[\frac{3}{4}(11S-21)|\nnn S|^2-\frac{5}{8}(\tri S)^2+\mathcal{C}_2+\mathcal{C}_3\right].
\end{equation*}
Combining Lemma \ref{lem a lower bound to B3}, we obtain
\begin{equation*}
\begin{aligned}
\int_M S(3S-4)(3S-5)(5S-9)&=2\int_M[\mathcal{B}_3-\frac{1}{8}S(3S-4)(45S^2-144S+116)\\
&\hspace{4.5em}+\frac{1}{8}(65S-166)|\nnn S|^2-\frac{5}{8}(\tri S)^2],
\end{aligned}
\end{equation*}
which completes the proof.
\end{proof}

\subsection{Proof of Theorem \ref{maintheorem}}
First, we give an upper bound of the integral of $|\nnn S|^2$.
\begin{lem}\label{tiduguji}
Let $M$ be a closed surface minimally immersed into $\mathbb{S}^N(1)$. If $\frac{5}{3}\le S\le\frac{9}{5}$ and $S_{\min}=\inf_{p\in M}S(p)$, then
$$\int_M|\nnn S|^2\le\int_M S(3S-4)(S-S_{\min}).$$
\end{lem}
\begin{proof}
Since $\text{div}(S\nnn S)=S\tri S+\l\nnn S,\nnn S\r$, then using \eqref{TWINST} we obtain
\begin{align*}
\int_M|\nnn S|^2&=\int_M-S\tri S\\
&=\int_M-S(2\mathcal{B}_1-S(3S-4))\\
&=\int_M(-2S\mathcal{B}_1+S^2(3S-4))\\
&\le-2S_{\min}\int_M\mathcal{B}_1+\int_M S^2(3S-4)\\
&=\int_M S(3S-4)(S-S_{\min}),
\end{align*}
which proves the lemma.
\end{proof}
\hspace{-2em}By Theorems \ref{thm gap s=1} and \ref{thm gap s=2}, Theorem \ref{maintheorem} (i)-(ii) is true. It remains to prove Theorem \ref{maintheorem} (iii). Notice that
\begin{equation*}
\begin{aligned}
\int_M(\tri S)^2&=2\int_M\mathcal{B}_1\tri S-\int_M S(3S-4)\tri S\\
&=2\int_M\mathcal{B}_1\tri S-3\int_M S^2\tri S+4\int_M S\tri S\\
&=2\int_M\mathcal{B}_1\tri S+\int_M(6S-4)|\nnn S|^2.
\end{aligned}
\end{equation*}
By the divergence theorem and (\ref{SecEq}) we obtain
\begin{equation*}
\begin{aligned}
\frac{1}{2}\int_M \mathcal{B}_1\tri S
&=\frac{1}{2}\int_M S\tri\mathcal{B}_1\\
&=\int_M\left(-\frac{7}{2}|\nnn S|^2+\frac{11}{4}S|\nnn S|^2-\frac{1}{4}S^2(3S-4)(9S-14)+S\mathcal{B}_2\right)\\
&\le\int_M\left((\frac{11}{4}S-\frac{7}{2})|\nnn S|^2-\frac{1}{4}S^2(3S-4)(9S-14)+S_{\max}(\frac{1}{4}S(3S-4)(9S-14)-\frac{1}{2}|\nnn S|^2)\right).
\end{aligned}
\end{equation*}
Then
\begin{equation*}
2\int_M\mathcal{B}_1\tri S\le\int_M(11S-14-2S_{\max})|\nnn S|^2-\int_M S(3S-4)(9S-14)(S-S_{\max}).
\end{equation*}
Combining these calculations gives that
\begin{equation}\label{xiayihang}
\int_M(\tri S)^2\le\int_M(17S-18-2S_{\max})|\nnn S|^2-\int_M S(3S-4)(9S-14)(S-S_{\max}).
\end{equation}
Inserting \eqref{xiayihang} into Theorem \ref{thm An integral inequality} and using Lemma \ref{tiduguji} give that
\begin{equation*}
\begin{aligned}
0&\ge\int_M S(3S-4)[-\frac{1}{2}(3S-5)(5S-9)+\frac{5}{8}(9S-14)(S-S_{\max})]\\
&\hspace{2.5em}+\int_M(-\frac{19}{8}S-\frac{9}{2}+\frac{5}{4}S_{\max})|\nnn S|^2\\
&\ge\int_M S(3S-4)[-\frac{1}{2}(3S-5)(5S-9)+\frac{5}{8}(9S-14)(S-S_{\max})]\\
&\hspace{2.5em}+\int_M(-\frac{19}{8}S_{}-\frac{9}{2}+\frac{5}{4}S_{\max})S(3S-4)(S-S_{\min})\\
&\ge\int_M S(3S-4)[-\frac{1}{2}(3S-5)(5S-9)+\frac{5}{8}(9S-14)(S-S_{\max})]\\
&\hspace{2.5em}+\int_M(-\frac{19}{8}S_{\max}-\frac{9}{2}+\frac{5}{4}S_{\max})S(3S-4)(S-S_{\min})\\
&=\int_M S(3S-4)[\frac{1}{2}(3S-5)(9-5S)+\frac{5}{8}(9S-14)(S-S_{\max})+(S-S_{\min})(-\frac{9}{8}S_{\max}-\frac{9}{2})]\\
&\eqqcolon\int_M S(3S-4)\frac{\mathcal{A}}{8},
\end{aligned}
\end{equation*}
where $\mathcal{A}=4(3S-5)(9-5S)+5(9S-14)(S-S_{\max})+(S-S_{\min})(-9S_{\max}-36)$. Here, with the condition $\frac{5}{3}\leq S\leq \frac{9}{5}$, we used the relation
\begin{equation*}
-\frac{19}{8}S_{}-\frac{9}{2}+\frac{5}{4}S_{\max}<0.
\end{equation*}
Assume that $\mathcal{A}>0$. Then, with $S>\frac{4}{3}$, we have
\begin{equation*}
0\ge\int_M S(3S-4)\mathcal{A}>0,
\end{equation*}
which creates a contradiction. Hence, $\mathcal{A}$ cannot be identically positive. To estimate $\mathcal{A}$, we have
\begin{equation*}
\begin{aligned}
\mathcal{A}&=4(3S-5)(9-5S)+5(S-S_{\max})(9S-14)-(S-S_{\min})(9S_{\max}+36)\\
&\ge4(3S-5)(9-5S)+5(S_{\min}-S_{\max})(9S_{\max}-14)-(S_{\max}-S_{\min})(9S_{\max}+36)\\
&=4(3S-5)(9-5S)-(S_{\max}-S_{\min})(54S_{\max}-34).
\end{aligned}
\end{equation*}
Suppose that $S_{\max}-S_{\min}<\varepsilon$, where $\frac{5}{3}\le S\le\frac{9}{5}$. Then
\begin{equation*}
\begin{aligned}
\mathcal{A}&>4(3S-5)(9-5S)-\varepsilon(54S_{\max}-34)\\
&>4(3S-5)(9-5S)-\varepsilon(54(S_{\min}+\varepsilon)-34)\\
&>4(3S_{\min}-5)(9-5S_{\max})-\varepsilon(54S_{\min}-34)-54\varepsilon^2\\
&>4(3S_{\min}-5)(9-5(S_{\min}+\varepsilon))-\varepsilon(54S_{\min}-34)-54\varepsilon^2\\
&=4(3S_{\min}-5)(9-5S_{\min})+\varepsilon(134-114S_{\min})-54\varepsilon^2.
\end{aligned}
\end{equation*}
Then we obtain that $$\mathcal{A}_1\coloneqq4(3S_{\min}-5)(9-5S_{\min})+\varepsilon(134-114S_{\min})-54\varepsilon^2$$ cannot be identically nonnegative. Let $$\mathcal{F}=(134-114S_{\min})^2+864(3S_{\min}-5)(9-5S_{\min}).$$ If
\begin{equation*}
\frac{(134-114S_{\min})-\sqrt{\mathcal{F}}}{108}\leq\varepsilon\leq\frac{(134-114S_{\min})+\sqrt{\mathcal{F}}}{108}\eqqcolon\varepsilon_0,
\end{equation*}
then $\mathcal{A}>\mathcal{A}_1\geq0$, which creates a contradiction. Choosing $\varepsilon=\varepsilon_0$, one has $$S_{\max}-S_{\min}\geq\varepsilon_0.$$ Therefore, we proved Theorem \ref{maintheorem}.\qed

\begin{acknow}
The authors are grateful to the reviewers for their insightful comments that significantly enhanced this work. The authors would like to thank Professor Zizhou Tang (Nankai University) for his encouragements and support. Finally, the authors want to thank Haiyang Wang for his discussions when he was a graduate student in Beijing Normal University.
\end{acknow}

%%%%%%%%%%%%%%%%%%%%%%%

\end{document}